\documentclass[a4paper,12pt]{amsart}

\usepackage{amsmath}
\usepackage{amssymb}
\usepackage{mathrsfs}
\usepackage{enumerate}
\usepackage{ifthen}
\usepackage{graphicx}
\usepackage{caption}
\usepackage[T1]{fontenc} 

\nonstopmode \numberwithin{equation}{section}
\newtheorem{theorem}{Theorem}[section]
\newtheorem{corollary}{Corollary}[section]

\theoremstyle{remark}
\theoremstyle{definition}
\newtheorem{remark}{Remark}[section]

\theoremstyle{plain}
\numberwithin{equation}{section}
\numberwithin{theorem}{section}

\newtheorem{innercustomthm}{Theorem}
\newenvironment{customthm}[1]
  {\renewcommand\theinnercustomthm{#1}\innercustomthm}
  {\endinnercustomthm}

\newcounter{minutes}
\divide\time by 60
\newcounter{hours}
\multiply\time by 60
\addtocounter{minutes}{-\time}

\begin{document}

\title{An application of Schur algorithm to variability regions of certain analytic functions-II}

\author{Md Firoz Ali}
\address{Md Firoz Ali,
Department of Mathematics
NIT Durgapur,
Mahatma Gandhi Avenue,
Durgapur - 713209,
West Bengal, India.}
\email{firoz.ali@maths.nitdgp.ac.in}

\author{ Vasudevarao Allu}
\address{ Vasudevarao Allu,
Discipline of Mathematics,
School of Basic Sciences,
Indian Institute of Technology  Bhubaneswar,
Argul, Bhubaneswar, PIN-752050, Odisha (State),  India.}
\email{avrao@iitbbs.ac.in}

\author{Hiroshi Yanagihara}
\address{Hiroshi Yanagihara,
    Department of Applied Science,
	Faculty of Engineering,
	Yamaguchi University,
	Tokiwadai, Ube 755,
	Japan}
\email{hiroshi@yamaguchi-u.ac.jp}

\subjclass[2010]{Primary 30C45, 30C75}
\keywords{Analytic functions, univalent functions, convex functions, starlike functions, Banach space, norm, Schur algorithm, subordination, variability region.}

\def\thefootnote{}
\footnotetext{ {\tiny File:~\jobname.tex,
printed: \number\year-\number\month-\number\day,
          \thehours.\ifnum\theminutes<10{0}\fi\theminutes }
} \makeatletter\def\thefootnote{\@arabic\c@footnote}\makeatother

\begin{abstract}
We continue our study on variability regions in \cite{Ali-Vasudevarao-Yanagihara-2018}, where the authors determined  the region of variability $V_\Omega^j (z_0, c ) = \{ \int_0^{z_0} z^{j}(g(z)-g(0))\, d z : g({\mathbb D}) \subset \Omega, \; (P^{-1} \circ g) (z) = c_0 +c_1z + \cdots + c_n z^n + \cdots \}$ for each fixed $z_0 \in {\mathbb D}$, $j=-1,0,1,2, \ldots$ and $c  = (c_0, c_1 , \ldots , c_n) \in \mathbb{C}^{n+1}$, when $\Omega\subsetneq\mathbb{C}$ is a convex domain, and $P$ is a conformal map of the unit disk ${\mathbb D}$ onto $\Omega$. In the present article, we first show that in the case $n=0$, $j=-1$ and $c=0$, the result obtained in \cite{Ali-Vasudevarao-Yanagihara-2018} still holds when one assumes only that $\Omega$ is starlike  with respect to $P(0)$. Let $\mathcal{CV}(\Omega)$ be the class of analytic functions $f$ in ${\mathbb D}$ with $f(0)=f'(0)-1=0$ satisfying   $1+zf''(z)/f'(z) \in {\Omega}$. As applications we determine variability regions of $\log f'(z_0)$ when $f$ ranges over $\mathcal{CV}(\Omega)$ with or without the conditions $f''(0)= \lambda$ and $f'''(0)= \mu$. Here $\lambda$ and $\mu$ are arbitrarily preassigned values. By choosing particular $\Omega$, we obtain the precise variability regions of $\log f'(z_0)$ for other  well-known subclasses of analytic and univalent functions.
\end{abstract}

\thanks{}

\maketitle
\pagestyle{myheadings}
\markboth{Md Firoz Ali, Vasudevarao Allu  and  Hiroshi Yanagihara}{Variability regions of certain analytic and univalent functions}


\section{Introduction}
Let ${\mathbb C}$ be the complex plane. For $c \in {\mathbb C}$ and $r > 0$, let ${\mathbb D}(c,r) := \{ z \in {\mathbb C} : |z-c| < r \}$ and $\overline{\mathbb D}(c,r) := \{ z \in {\mathbb C} : |z-c| \leq r \}$. In particular, we denote the unit disk by ${\mathbb D} := {\mathbb D}(0,1)$.
Let ${\mathcal A}({\mathbb D})$ be the class of analytic functions in the unit disk ${\mathbb D}$ endowed with the topology of uniform convergence on every compact subset of ${\mathbb D}$. Denote by $\mathcal{A}_0$, functions $f$ in ${\mathcal A}({\mathbb D})$, normalized by $f(0) = f'(0)-1 = 0$. Further, let $\mathcal{S}$ denote the standard subclass of $\mathcal{A}_0$  of normalized univalent functions in ${\mathbb D}$.
A function $f$ in $\mathcal{A}_0$ is called starlike (resp. convex) if $f$ is univalent and $f(\mathbb{D})$ is starlike with respect to $0$ (resp. convex). Let $\mathcal{S}^*$ and $\mathcal{CV}$ denote the classes of starlike and convex functions respectively. It is well-known that a function $f\in\mathcal{A}_0$ is in $\mathcal{S}^*$ (resp. $\mathcal{CV}$) if, and only if, ${\rm Re\,}\left(zf'(z)/f(z)\right)>0$ (resp. ${\rm Re} \, \{ zf''(z)/f'(z) \} +1 > 0$) for $z\in\mathbb{D}$.

\bigskip
Let ${\mathcal F}$ be a subclass of ${\mathcal A}({\mathbb D})$ and $z_0 \in {\mathbb D}$. Then upper and lower estimates of the form
$$
 M_1 \leq |f'(z_0)| \leq M_2 ,
 \qquad m_1 \leq \mbox{\rm Arg}\, f'(z_0) \leq m_2
\quad
\mbox{for all } \; f \in {\mathcal F}
$$
are respectively called  distortion  and rotation theorems at $z_0$ for ${\mathcal F}$, where $M_j$ and $m_j$ ($j=1,2$) are some non-negative constants. These estimates deal only with absolute value or argument of $f'(z_0)$. If one wants to study the complex value $f'(z_0)$ itself, it is necessary to consider the variability region of $f'(z_0)$ when $f$ ranges over ${\mathcal F}$, i.e., the set $\{ f'(z_0) : f \in {\mathcal F } \}$. For example \cite[Chapter 2, Exercise 10, 11 and 13]{Duren-book}, it is known that
$$
  \{ \log f'(z_0) : f \in \mathcal{CV} \}
  =
  \left\{ \log \frac{1}{(1-z)^2 } : |z| \leq |z_0| \right\} .
$$

\bigskip
For $f \in \mathcal{CV}$, an easy consequence of the Schwarz's lemma is that that $|f''(0)| \leq 2$. For fixed $z_0 \in {\mathbb D}$ and $\lambda \in \overline{\mathbb D}$, Gronwall \cite{Gronwall-1920} obtained the sharp lower and upper estimates for $|f'(z_0)|$, when $f \in \mathcal{CV}$ satisfies the additional condition $f''(0) = 2 \lambda$ (see also \cite{Finkelstein-1967}). Let
$$
 \widetilde{V} (z_0, \lambda)
 = \{ \log f'(z_0) : f \in \mathcal{CV} \; \mbox{and} \;
  f''(0) = 2 \lambda \} .
$$
If $| \lambda | =1 $, then by Schwarz's lemma, for $f \in \mathcal{CV}$  the condition $f''(0) = 2 \lambda $ forces $f(z) \equiv z/(1-\lambda z)$, and hence $\widetilde{V} (z_0, \lambda) = \{ \log 1/(1-\lambda z_0)^2  \}$.
Since $\widetilde{V}(e^{-i \theta}z_0, e^{i \theta }\lambda) = \widetilde{V}(z_0, \lambda )$ for all $\theta \in {\mathbb R}$, without loss of generality we may assume that $0 \leq \lambda < 1$. In 2006, Yanagihra \cite{Yanagihra:convex} obtained the following extension of Gronwall's \cite{Gronwall-1920} result.

\begin{customthm}{A}\label{theorem-A}
For any $z_0 \in {\mathbb D} \backslash \{ 0 \}$ and $0 \leq \lambda < 1$
the set $\widetilde{V} (z_0, \lambda) $
is a convex closed Jordan domain surrounded by the curve
\begin{eqnarray*}
    && ( - \pi , \pi ] \ni \theta \mapsto
\\
&&
    - \left(
    1 - \frac{\lambda \cos ( \theta /2)}{\sqrt{1 - \lambda^2 \sin^2 ( \theta /2) }}
    \right)
    \log
    \left\{
    1  - \frac{e^{i \theta /2} z_0}
    {i \lambda \sin ( \theta /2) - \sqrt{1 - \lambda^2 \sin^2 ( \theta /2) } }
    \right\}
\\
    & & \; \; \;  \; -
    \left(
    1 + \frac{\lambda \cos ( \theta /2)}{\sqrt{1 - \lambda^2 \sin^2 ( \theta /2) }}
    \right)
    \log
    \left\{
    1  - \frac{e^{i \theta /2} z_0}
    {i \lambda \sin ( \theta /2) + \sqrt{1 - \lambda^2 \sin^2 ( \theta /2) } }
    \right\} .
\end{eqnarray*}
\end{customthm}

Theorem \ref{theorem-A}, can be equivalently written as follows.

\begin{customthm}{B}\label{theorem-B}
Let ${\mathbb H} = \{ w \in {\mathbb C} : \text{\rm Re} \, w > 0 \}$. For any $z_0 \in {\mathbb D} \backslash \{ 0 \}$ and $0 \leq \lambda < 1$, the variability region
$$
\left\{
     \int_0^{z_0} \frac{g(\zeta )-g(0)}{\zeta} \, d \zeta
     : \; g \in {\mathcal A}({\mathbb D}) \; \mbox{with} \;
     g(0)=1,\; g'(0)= 2 \lambda,\; g({\mathbb D}) \subset {\mathbb H}
  \right\}
$$
coincides with the same convex closed Jordan domain as in Theorem \ref{theorem-A}.
\end{customthm}

Theorem \ref{theorem-A} is a direct consequence of Theorem \ref{theorem-B} if we choose $g(z) = 1 + zf''(z)/f'(z)$. For similar results, we refer to \cite{Ponnusamy-Vasudevarao-2007,Ponnusamy-Vasudevarao-Yanagihara-2008,Ul-Haq-2014,Yanagihara:bounded,Yanagihra:unif_convex} and the references therein.

\bigskip
Recently, the present authors \cite{Ali-Vasudevarao-Yanagihara-2018} extended Theorem \ref{theorem-B} to the most general setting. 
\bigskip

Let $\Omega $ be a simply connected domain in ${\mathbb C}$ with $\Omega \not= {\mathbb C}$, and $P$ be a conformal map of ${\mathbb D}$ onto $\Omega$. Let ${\mathcal F}_\Omega$ be the class of analytic functions $g$  in ${\mathbb D}$ with $g( {\mathbb D}) \subset \Omega$. Then the map $P^{-1} \circ g$ maps ${\mathbb D}$ into ${\mathbb D}$.
For $c  = (c_0,c_1, \ldots , c_n ) \in {\mathbb C}^{n+1}$, let
$$
   {\mathcal F}_\Omega (c )
   =
   \{ g \in {\mathcal F}_\Omega : \;
   (P^{-1} \circ g) (z) = c_0 + c_1z+\cdots +c_nz^n + \cdots \;
   \text{in} \; {\mathbb D}  \}.
$$
Let $H^\infty ({\mathbb D})$ be the Banach space of analytic functions $f$ in ${\mathbb D}$ with the norm $\| f \|_\infty = \sup_{z \in {\mathbb D}} |f(z)|$, and $H_1^\infty ({\mathbb D})$ be the closed unit ball of $H^\infty ({\mathbb D})$, i.e., $H_1^\infty ({\mathbb D}) = \{ \omega \in H^\infty ({\mathbb D}) : \| \omega \|_\infty \leq 1 \}$. We note that the coefficient body ${\mathcal C}(n)$ defined by
\begin{align*}
 {\mathcal C}(n)=
&
 \{ c  = (c_0,c_1, \ldots , c_n ) \in {\mathbb C}^{n+1}: \; \mbox{where there exists $\omega \in H_1^\infty ({\mathbb D})$}\\
&\quad
\mbox{such that } \; \omega (z) = c_0+c_1 z + \cdots + c_n z^n  + \cdots \; \mbox{in} \; {\mathbb D}
 \}
\end{align*}
is a compact and convex subset of ${\mathbb C}^{n+1}$. The coefficient body ${\mathcal C}(n)$ has been completely  characterized Schur \cite{Schur-1917,Schur-1986}. For a detailed treatment, we refer to \cite[Chapter I]{Foias-Brazho-book} and \cite[Chapter 1]{Bakonyi-Constantinescu-1992}.

\bigskip
We call $c  = (c_0, \ldots , c_n )$ the Carath\'{e}odory data of length $n+1$.
For a given Carath\'{e}odory data $c =(c_0,\ldots ,c_n) \in {\mathbb C}^{n+1}$,
the Schur parameter $\gamma = (\gamma_0 , \ldots , \gamma_k)$,
$k=0,1, \ldots , n$ is defined as follows. 
\bigskip

For $j=0,1,\ldots$, define recursively
$c^{(j)} = (c_0^{(j)},c_1^{(j)} \ldots , c_{n-j}^{(j)})$ and $\gamma_j = c_0^{(j)}$ by
\begin{equation}\label{eq:c_p_j+1}
 c_0^{(j)} = \frac{c_1^{(j-1)}}{1-|\gamma_{j-1}|^2}, \quad
 c_p^{(j)} = \frac{c_{p+1}^{(j-1)}
 + \overline{\gamma_{j-1}} \sum_{\ell =1}^p c_{p-\ell}^{(j)} c_\ell^{(j-1)}}{1-|\gamma_{j-1}|^2}
    \quad (1 \leq p \leq n- j),
\end{equation}
with $c ^{(0)} =c = (c_0,\ldots ,c_n)$. In the $j$th step ($j=0,1,\ldots$),
if $|\gamma_j| > 1$, then we put $k=j$ and $\gamma = ( \gamma_0, \ldots , \gamma_j )$;
if $|\gamma_j | =1$, then we put $k=n$, and for $p=j+1, \ldots , n$, we take
$$
\gamma_p =
\begin{cases}
   \infty , & \mbox{if} \; \, c_{p-j}^{(j)} \not= 0 \\[2mm]
   0, & \mbox{if} \; \, c_{p-j}^{(j)} =0;
\end{cases}
$$
if $|\gamma_j | < 1$ then we proceed to $(j+1)$th step and so on.
Applying this procedure recursively we obtain  the Schur parameter
$\gamma = ( \gamma_0 , \ldots , \gamma_k )$, $k=0, \ldots , n$ of $c =(c_0, \ldots , c_n)$.

\bigskip
When $|\gamma_0|<1, \ldots , |\gamma_n| < 1$, $c=(c_0, \ldots , c_n )= c^{(0)}$ and $\gamma = (\gamma_0, \ldots , \gamma_n )$ are determined uniquely each other. For an explicit representation of $\gamma$ in terms of $c$, we refer to \cite{Schur-1917,Schur-1986}. For given $c=(c_0, \ldots , c_n) \in {\mathbb C}^{n+1}$, Schur \cite{Schur-1917,Schur-1986} proved that $c  \in \text{Int} \, {\mathcal C}(n)$, $c  \in \partial {\mathcal C}(n)$ and $c  \not\in {\mathcal C}(n)$ are respectively equivalent to the conditions $\mathrm{\mathbf{(C1)}}$ $k=n$ and $|\gamma_i|<1$ for $i=1,2,\ldots,n$ $\mathrm{\mathbf{(C2)}}$, $k=n$ and $|\gamma_0|<1, \ldots , |\gamma_{i-1}|<1$, $|\gamma_i|=1$, $\gamma_{i+1}= \cdots = \gamma_n=0$ for some $i=0, \ldots , n$ and $\mathrm{\mathbf{(C3)}}$ the hypotheses that either $\mathrm{\mathbf{(C1)}}$ or $\mathrm{\mathbf{(C2)}}$ does not hold. Furthermore, for $c  \in \text{Int} \, {\mathcal C}(n)$, the Schur parameter can be computed as follows. 
\bigskip

Let $\omega \in H_1^\infty ({\mathbb D})$ be such that $\omega (z) = c_0+ c_1 z + \cdots + c_n z^n + \cdots$. Define
$$
\omega_0(z) = \omega(z) \quad\mbox{and}\quad
\omega_k (z) = \displaystyle  \frac{\omega_{k-1} (z)-\omega_{k-1}(0)}{z(1-\overline{\omega_{k-1}(0) }\omega_{k-1}(z))}
\quad (k=1,2,\ldots,n).
$$
Then
$$
 \gamma_p = \omega_p(0), \quad \omega_p(z) = c_0^{(p)}+c_1^{(p)}z+ \cdots + c_{n-p}^{(p)}z^{n-p} + \cdots
$$
hold for $p=0,1, \ldots , n$. For a detailed proof, we refer to \cite[Chapter 1]{Foias-Brazho-book}.

\bigskip
For $a \in {\mathbb D}$, define $\sigma_a \in \mbox{Aut} ( {\mathbb D})$ by
$$
\sigma_a (z) = \frac{z+a}{1+ \overline{a}z},  \quad z \in {\mathbb D} .
$$
For $\varepsilon \in \overline{\mathbb D}$ and the Schur parameter $\gamma =( \gamma_0,\ldots , \gamma_n)$ of $c  \in \text{Int} \, {\mathcal C}(n)$, let
\begin{align}
\omega_{\gamma , \varepsilon }(z)
    =&  \sigma_{\gamma_0} ( z \sigma_{\gamma_1} ( \cdots z \sigma_{\gamma_{n}} ( \varepsilon z) \cdots )), \quad z \in {\mathbb D} ,
\label{def:extremal_omega}
\\
Q_{\gamma , j} (z, \varepsilon )
    =& \int_0^z \zeta^j \{ P(\omega_{\gamma,\varepsilon} (\zeta )) - P(c_0) \}\, d \zeta , \quad z \in {\mathbb D} \; \text{and} \;
    \varepsilon \in \overline{\mathbb D}.
\label{def:extremal_Q}
\end{align}
Then $\omega_{\gamma , \varepsilon } \in H_1^\infty ({\mathbb D})$ with Carath\'{e}odory data $c$, i.e., $\omega_{\gamma , \varepsilon }(z) = c_0+c_1z+\cdots + c_nz^n + \cdots$. By using the Schur algorithm, recently the present authors \cite{Ali-Vasudevarao-Yanagihara-2018} obtained the following general result for the region of variability.

\begin{customthm}{C}\cite{Ali-Vasudevarao-Yanagihara-2018}\label{thm:Main_theorem}
Let $n \in {\mathbb N} \cup \{ 0 \}$, $j \in \{-1,0, 1,2 , \ldots \}$, and $c  =(c_0, \ldots , c_n) \in {\mathbb C}^{n+1}$ be a Carath\'{e}odory data. Let $\Omega $ be a convex domain in ${\mathbb C}$ with $\Omega \not= {\mathbb C}$, and $P$ be a conformal map of ${\mathbb D}$ onto $\Omega$. For each fixed $z_0 \in {\mathbb D} \backslash \{0 \}$, let
\begin{equation*}
 V_\Omega^j (z_0, c )
 =
 \left\{
   \int_0^{z_0} \zeta^j \left( g( \zeta ) - g(0) \right) \, d \zeta
   : \; g \in {\mathcal F}_\Omega  (c )
  \right\} .
\end{equation*}

\begin{enumerate}[{\rm (i)}]

\item If $c  = (c_0, \ldots , c_n ) \in \text{\rm Int} \, {\mathcal C}(n)$
and $\gamma =(\gamma_0, \ldots , \gamma_n )$ be the Schur parameter of $c$,
then $Q_{\gamma , j}(z_0, \varepsilon )$ defined by (\ref{def:extremal_omega})
is a convex univalent function of $\varepsilon \in \overline{\mathbb D}$ and
$$
V_\Omega^j (z_0,c) =   Q_{\gamma , j}(z_0, \overline{\mathbb D} )
  := \{ Q_{\gamma , j}(z_0, \varepsilon ) : \varepsilon \in \overline{\mathbb D} \} .
$$
Furthermore,
$$
\int_0^{z_0} \zeta^j \{ g(\zeta ) - g(0) \} \, d \zeta = Q_{\gamma , j}(z_0, \varepsilon )
$$
for some $g \in {\mathcal F}_\Omega (c )$ and $\varepsilon \in \partial {\mathbb D}$
if, and only if, $g (z) \equiv P( \omega_{\gamma , \varepsilon } (z ))$.

\item If $c  \in \partial {\mathcal C}(n)$ and $\gamma =( \gamma_0, \ldots , \gamma_i, 0, \ldots , 0 )$
is  the Schur parameter of $c$, then $V_\Omega^j (z_0,c )$ reduces to a set consists of a single point $w_0$,
where
$$
w_0 = \int_0^{z_0} \zeta^j
 \{ P( \sigma_{\gamma_0} ( \zeta \sigma_{\gamma_1}(\cdots \zeta \sigma_{\gamma_{i-1}} (\gamma_i \zeta  ) \cdots )))- P(c_0) \} \, d \zeta.
$$

\item If $c  \not\in {\mathcal C}(n)$ then $V_\Omega^j (z_0,c ) = \emptyset$.

\end{enumerate}
\end{customthm}

In the present article, we first show that in the case $n=0$, $j=-1$ and $c=0$, the conclusion of Theorem \ref{thm:Main_theorem} holds  when one weakens the assumption that $\Omega$ is convex to that of starlikeness of $\Omega$ with respect to $P(0)$. We then present several applications of Theorems \ref{theorem-B} and \ref{thm:of_order_zero} to obtain the precise variability region of different quantities for several well-known subclasses of analytic and univalent functions. We also obtain certain subordination results.

\section{Main Results}

\begin{theorem}\label{thm:of_order_zero}
Let $b \in {\mathbb C}$, $z_0 \in {\mathbb D} \backslash \{ 0 \}$
and $\Omega$ be a starlike domain with respect to $b$ satisfying $\Omega \not= {\mathbb C}$.
Let $P$ be a conformal map of ${\mathbb D}$ onto $\Omega$ with $P(0) =b$. Then the region of variability
$$
V_\Omega^{-1} (z_0,0)  = \left\{ \int_0^{z_0} \frac{g(\zeta ) - b}{\zeta } \, d \zeta : g \in {\mathcal F}_\Omega , \; g(0) = b \right\}
$$
is a convex closed Jordan domain, and coincides with the set $K ( \overline{\mathbb D}(0, |z_0| ) )$,
where $K(z) = \int_0^z \zeta^{-1} (P(\zeta ) -b) \, d \zeta $. Furthermore, for $| \varepsilon | =1$ and
$g \in {\mathcal F}_\Omega$ with $g(0) = b$, the relation
$\int_0^{z_0} \zeta^{-1}(g(\zeta ) - b) \, d \zeta = K(\varepsilon z_0)$ holds if, and only if,
$g(z) \equiv P ( \varepsilon z )$.
\end{theorem}

\begin{proof}
Let $g \in {\mathcal A} ({\mathbb D})$ be such that $g(0)=b$ and $g({\mathbb D}) \subset \Omega$.
Then $g \prec P$, i.e., $g$ is subordinate to $P$. By using a result of Suffridge \cite{Suffridge-1970}, we may conclude
$$
\int_0^z \frac{ g(\zeta) -b }{\zeta} \, d \zeta \prec  K(z) := \int_0^z \frac{ P(\zeta) -b }{\zeta} \, d \zeta.
$$
Thus there exists  $\omega \in H_1^\infty ({\mathbb D})$ with $\omega(0)=0$ and
$\int_0^z \zeta^{-1} \{g(\zeta) -b \}\, d \zeta = K(\omega(z))$. From this it follows that
$$
  V_\Omega^{-1} (z_0,0)  \subset
      \{ K(\omega (z_0) ) : \omega \in H_1^\infty ({\mathbb D}) \; \mbox{and} \; \omega(0) = 0 \}
      = K\left( \overline{\mathbb D}(0,|z_0| )\right) .
$$
For $\varepsilon \in \overline{\mathbb D}$, let $g_\varepsilon (z) = P(\varepsilon z)$.
Then $g_\varepsilon (0) = P(0) = b$ and $g_\varepsilon ({\mathbb D}) = P({\mathbb D})=\Omega$. Therefore
$$
K(\varepsilon z_0)=\int_0^{\varepsilon z_0} \frac{ P(\zeta) -b }{\zeta} \, d \zeta
= \int_0^{z_0} \frac{ g_\varepsilon(\zeta) -b }{\zeta} \, d \zeta \in V_\Omega^{-1} (z_0, 0),
$$
and hence $K(\overline{\mathbb D}(0,|z_0|)) \subset V_\Omega^{-1} (z_0,0)$.

\bigskip
We now deal with the uniqueness. Suppose that
\begin{equation}\label{eq:boundary_point}
\int_0^{z_0} \frac{g( \zeta ) - b}{\zeta} \, d \zeta = K(\varepsilon z_0)
\end{equation}
holds for some $g$ with $g(0)=b$ and $g({\mathbb D}) \subset \Omega $, and $| \varepsilon | =1$.
Then there exists $\omega \in H_1^\infty ({\mathbb D})$ with $\omega(0)=0$ such that
$\int_0^z \zeta^{-1} \{g(\zeta) -b \} \, d \zeta = K(\omega (z))$.
From (\ref{eq:boundary_point}) we have $K( \omega(z_0)) = K ( \varepsilon z_0 )$.
Since $K$ is a convex univalent function,  $\omega(z_0) = \varepsilon z_0$.
It follows from Schwarz's lemma that $\omega (z) \equiv \varepsilon z$.
Consequently $g(z) \equiv P(\varepsilon z)$.
\end{proof}


\subsection{The Class $\mathcal{CV} (\Omega)$}


Suppose that $\Omega$ is a simply connected domain with
$1 \in \Omega$. Define
$$
\mathcal{CV} (\Omega) =
 \left\{ f \in {\mathcal A}_0({\mathbb D}) :  1 + z \frac{f''(z)}{f'(z)} \in \Omega \; \text{for all}
 \; z\in{\mathbb D} \right\} .
$$
Let $P$ be the conformal map of $\mathbb{D}$ onto $\Omega$ with $P(0)=1$.
Then for each $f\in\mathcal{CV}(\Omega)$, we have $1+zf''(z)/f'(z)\prec P$. For $\alpha\in\mathbb{R}$,
let $\mathbb{H}_{\alpha}:=\{z\in\mathbb{C}: {\rm Re\,} z>\alpha\}$ and $\mathbb{H}_{0}=\mathbb{H}$.
When $\Omega = {\mathbb H}$ and $P(z)=(1+z)/(1-z)$, $\mathcal{CV}({\mathbb H}) = \mathcal{CV}$
is the well-known class of normalized convex functions in ${\mathbb D}$.
If $\Omega \subset {\mathbb H}$, then $\mathcal{CV}(\Omega )$ is a subclass of $\mathcal{CV}$.
For $0 \leq \alpha < 1$, $\mathcal{CV}(\alpha):=\mathcal{CV}(\mathbb{H}_{\alpha})$
is the class of convex functions of order $\alpha$.
In this case, $P(z) = \{1+ (1-2 \alpha )z\}/(1-z)$.
If $0 < \beta \leq 1$, then $\mathcal{CV}_{\beta}:=\mathcal{CV}( \{ w \in {\mathbb C} : | \arg \, w| < \pi \beta /2\})$
is the class of strongly convex functions of order $\beta$ and $P$ is given by $P(z) = \{(1+z)/(1-z)\}^\beta$.

\bigskip
As an application of Theorem \ref{thm:of_order_zero} we determine the variability region of
$\log f'(z_0)$ when $f$ ranges over $\mathcal{CV}(\Omega)$.

\begin{theorem}\label{thm-p001}
Let $\Omega$ be a starlike domain with respect to $1$, and $P$ be a conformal map of ${\mathbb D}$ onto $\Omega$
with $P(0)=1$. Then for each fixed $z_0 \in {\mathbb D} \backslash \{ 0 \}$, the region of variability
$$
V_{\mathcal{CV}(\Omega)}(z_0):= \{ \log f'(z_0 ) : f \in \mathcal{CV}( \Omega ) \}
$$
is a convex closed Jordan domain, and coincides with the set $K( \overline{\mathbb D}(0,|z_0|))$,
where $K(z) = \int_0^z \zeta^{-1} (P( \zeta )- 1) \, d \zeta $ is a convex univalent function in
${\mathbb D}$. Furthermore, $\log f'(z_0) = K( \varepsilon z_0)$ for some $|\varepsilon | =1$ and $f \in \mathcal{CV}( \Omega )$
if, and only if, $f(z) = \varepsilon^{-1}F(\varepsilon z)$, where $F(z)= \int_0^z e^{K(\zeta )} \, d \zeta $.
\end{theorem}

\begin{proof}
Let $c = 0 \in {\mathbb C}^1$ be a given Carath\'{e}odry data of length one.
Then ${\mathcal F}_\Omega (0)= \{ g \in {\mathcal A}({\mathbb D}):
g({\mathbb D}) \subset \Omega \; \text{and} \; (P^{-1} \circ g) (0) = 0 \}$.
It is easy to see that the map
$$
 \mathcal{CV}( \Omega )  \ni f  \mapsto g(z) = 1 + z \frac{f''(z)}{f'(z)}  \in {\mathcal F}_\Omega  (0)
$$
is bijective. Indeed, since $g(z) = 1+ zf''(z)/f'(z)$ is analytic in ${\mathbb D}$,
$f'(z)$ does not have zeros in ${\mathbb D}$, and so 
\begin{equation}\label{eq:p-005}
 \log f'(z)=  \int_0^z \zeta^{-1} (g( \zeta )-1)\, d \zeta ,
\end{equation}
where $\log f'$ is a single valued branch of the logarithm of $f'$ with $\log f'(0) = 0$.
The conclusions now follows from Theorem \ref{thm:of_order_zero} and (\ref{eq:p-005}).
\end{proof}

As an application of Theorems \ref{thm:Main_theorem} we determine
the variability region of $\log f'(z_0)$, when $f$ ranges over $\mathcal{CV}(\Omega)$
with the conditions $f''(0) = 2 \lambda $ and $f'''(0) = 6 \mu$. Here $z_0 \in {\mathbb D} \backslash \{ 0 \}$,
$\lambda, \mu \in {\mathbb C}$ are arbitrarily preassigned values. By letting $\Omega$ be one of the particular domains
mentioned in the above, we can determine variability regions of $\log f'(z_0)$ for various subclasses of $\mathcal{CV}$.

\bigskip
Let $\Omega$ be a simply connected domain with $\Omega \not= {\mathbb C}$,
and $P$ be a conformal map of ${\mathbb D}$ onto $\Omega$
with $P(z)= \alpha_0 + \alpha_1 z + \alpha_2z^2 + \cdots $.
Let $g$ be an analytic function in ${\mathbb D}$
with $g(z)= b_0+b_1 z + b_2 z^2 +\cdots $
satisfying $g({\mathbb D}) \subset \Omega $.
For simplicity we assume that $P(0)=g(0)$, i.e., $\alpha_0 =b_0$. Let
$$
\omega (z) = (P^{-1}\circ g)(z) = c_0 + c_1 z + c_2 z^2 + \cdots ,\quad z \in {\mathbb D} .
$$
Then
\begin{equation}\label{eq:relation_between_c_and_b}
  c_0 = 0, \quad
  c_1 =  \frac{b_1}{\alpha_1}, \quad
  c_2 = \frac{\alpha_1^2 b_2 - \alpha_2 b_1^2}{\alpha_1^3} .
\end{equation}
By Schwarz's lemma $|b_1| \leq |\alpha_1|$, with equality 
if, and only if, $g(z) = P(\varepsilon z)$ for some $\varepsilon \in \partial {\mathbb D}$.
Let $\gamma = (\gamma_0,\gamma_1,\gamma_2)$ be the Schur parameter of the
Carath\'{e}odory data $c = (0,c_1,c_2)$. Then $\gamma_0 = \omega (0) = c_0 = 0$, and
\begin{align*}
 \omega_1(z) =& \frac{\omega(z)}{z}, \quad \gamma_1 = \omega_1 (0), \\
  \omega_2(z) =& \frac{\omega_1 (z)-\gamma_1}{z(1-\overline{\gamma_1}\omega_1(z))}, \quad \gamma_2 = \omega_2 (0).
\end{align*}
A simple computation shows that
\begin{equation}\label{eq:relation_between_gamma_and_c}
  \gamma_0 = 0 , \quad
  \gamma_1 = c_1 = \frac{b_1}{\alpha_1} , \quad
  \gamma_2 = \frac{c_2}{1-|c_1|^2}
   =
   \frac{ \overline{\alpha_1}
 ( \alpha_1^2 b_2 - \alpha_2 b_1^2 )}
     {\alpha_1^2(|\alpha_1|^2 -|b_1|^2)}.
\end{equation}

\bigskip
For $f \in \mathcal{CV}( \Omega )$ and $k \in {\mathbb N}$, let $a_k(f) = f^{(k)}(0)/k!$.
Let $g(z) = 1 + zf''(z)/f'(z) = 1+b_1z +b_2z^2 +\cdots $. Then
\begin{equation}\label{eq:relation_between_b_and_a}
  b_1 = 2 a_2(f) \quad\mbox{and}\quad b_2 = 6 a_3(f) -4 a_2(f)^2.
\end{equation}
From (\ref{eq:relation_between_gamma_and_c}) and (\ref{eq:relation_between_b_and_a}), we have
\begin{equation}\label{eq:p-025}
  \gamma_0 = 0 , \quad
  \gamma_1 =  \frac{2 a_2(f) }{\alpha_1} ,  \quad
  \gamma_2 =
 \frac{ 2\overline{\alpha_1}
 \{ 3\alpha_1^2 a_3(f) - 2(\alpha_1^2+\alpha_2) a_2(f)^2 \}}
 {\alpha_1^2(|\alpha_1|^2-4|a_2(f)|^2)} .
\end{equation}
Let
$$
  {\mathcal A} (2 , \Omega)
  = \{ a_2(f) : f \in \mathcal{CV} ( \Omega )\}.
$$
Then by the Schwarz lemma,
${\mathcal A} (2 , \Omega) = \overline{\mathbb D}(0, |\alpha_1|/2)$.
For $f \in \mathcal{CV}(\Omega)$ and
$\lambda \in \partial {\mathcal A} (2 , \Omega)$,
$a_2 (f) = \lambda $
 if, and only if,
$f(z) \equiv \gamma_1^{-1}F(\gamma_1 z)$,
where $\gamma_1 = 2 \lambda / \alpha_1$.
By applying Theorem \ref{thm:Main_theorem} with $n=1$ and $j=-1$
we obtain the following generalization of Theorem \ref{theorem-A}.

\begin{theorem}\label{thm-p005}
Let $\Omega$ be a convex domain with $1 \in \Omega$ and $P$ be a conformal map of ${\mathbb D}$ onto
$\Omega$ with $P(z)=1+ \alpha_1 z + \cdots $. For $\lambda \in {\mathbb C}$ with $|\lambda | \leq |\alpha_1|/2 $
and $z_0 \in {\mathbb D} \backslash \{ 0 \}$, consider the variability region
$$
V_{\mathcal{CV}(\Omega)}(z_0,\lambda):=
\{ \log f'(z_0) : f \in \mathcal{CV}( \Omega ) \; \mbox{with} \; a_2(f) = \lambda \}.
$$
\begin{enumerate}[{\rm (i)}]

\item If $|\lambda |= |\alpha_1|/2 $, then $V_{\mathcal{CV}(\Omega)}(z_0,\lambda)$
reduces to a set consists of a single point $w_0$, where
$w_0 = \int_0^{z_0} \zeta^{-1} \{ P(\gamma_1 \zeta)- 1 \} \, d \zeta$ with $\gamma_1 = 2 \lambda /\alpha_1$.\\


\item If $|\lambda | < |\alpha_1|/2 $ then
$V_{\mathcal{CV}(\Omega)}(z_0,\lambda) = Q_{\gamma_1}(z_0, \overline{\mathbb D} )$,
where $\gamma_1 = 2 \lambda /\alpha_1$ and
$$
 Q_{\gamma_1}(z_0, \varepsilon ) = \int_0^{z_0} \zeta^{-1} \left\{ P \left( \zeta
 \frac{\varepsilon \zeta + \gamma_1} {1+ \overline{\gamma_1} \varepsilon \zeta} \right) -1 \right\} \, d \zeta
$$
is a convex, univalent and analytic function of $\varepsilon \in \overline{\mathbb D}$. Furthermore,
$$
 \log f'(z_0)  =  Q_{\gamma_1}(z_0, \varepsilon )
$$
holds for some $\varepsilon \in \partial {\mathbb D}$ and $f \in \mathcal{CV}( \Omega )$ with
$a_2(f) =  \lambda $, if, and only if,
$$
   f(z)  = \int_0^z  e^{Q_{\gamma_1}(\zeta , \varepsilon )} \, d \zeta , \quad z \in {\mathbb D}.
$$
\end{enumerate}
\end{theorem}

Next let
$$
  {\mathcal A} (3 , \Omega) =
  \{ (a_2(f), a_3(f) ) \in {\mathbb C}^2 : f \in \mathcal{CV}( \Omega )  \},
$$
and for $\lambda , \mu \in {\mathbb C}$, let $\gamma_1 := \gamma_1 (\lambda,\mu)$
and $\gamma_2 := \gamma_2 (\lambda , \mu )$ be given by
\begin{equation}\label{eq:p-010}
\gamma_1 = \frac{2 \lambda }{\alpha_1}
\end{equation}
and
\begin{equation}\label{eq:p-015}
\gamma_2 =
  \begin{cases}
  \displaystyle \frac{ 2\overline{\alpha_1} \{ 3\alpha_1^2 \mu - 2(\alpha_1^2+\alpha_2)\lambda^2 \}}
 {\alpha_1^2(|\alpha_1|^2-4|\lambda|^2)}, \quad & \text{if} \; |\gamma_1| < 1 , \\[2mm]
 0 , &   \text{if} \; |\gamma_1| =1  \; \text{and} \;  3\alpha_1^2 \mu = 2(\alpha_1^2+\alpha_2)\lambda^2 , \\[2mm]
  \infty , &   \text{if} \; |\gamma_1| = 1  \; \text{and} \;  3\alpha_1^2 \mu \not= 2(\alpha_1^2+\alpha_2)\lambda^2 .
  \end{cases}
\end{equation}
Then $(\lambda , \mu ) \in {\mathcal A} (3 , \Omega)$ if, and only if, one of the following conditions holds.
\begin{enumerate}[{\rm (a)}]

\item $|\gamma_1 (\lambda , \mu ) | = 1 $ and  $\gamma_2 ( \lambda , \mu )=0$.\\[-3mm]

\item $|\gamma_1 (\lambda , \mu ) | < 1 $ and  $ | \gamma_2 ( \lambda , \mu ) |=1$.\\[-3mm]

\item $|\gamma_1 (\lambda , \mu ) | < 1 $ and  $| \gamma_2 ( \lambda , \mu )| < 1$.

\end{enumerate}
In case (a), for $f \in \mathcal{CV}( \Omega )$, $(a_2(f), a_3(f)) = (\lambda, \mu )$
holds if, and only if, $g(z) = P(\gamma_1 z )$, i.e., $f(z) = \gamma_1 F(\gamma_1 z)$,
where $\gamma_1 = \gamma_1(\lambda , \mu )$. Similarly, in case (b), for $f \in \mathcal{CV}( \Omega )$,
$(a_2(f), a_3(f)) = (\lambda , \mu ) $ holds if, and only if,
$g(z) = P(z \sigma_{\gamma_1}( \gamma_2 z) )$, i.e.,
$$
f(z) = \int_0^z
\exp\left[ \int_0^{\zeta_1} \zeta_2^{-1} \{ P(\zeta_2 \sigma_{\gamma_1}(
\gamma_2 \zeta_2) )-1\} \, d \zeta_2 \right] \, d \zeta_1.
$$
We note that $(\lambda , \mu ) \in \partial {\mathcal A} (3 , \Omega)$ if, and only if, either (a) or (b) holds.

\bigskip
Suppose that (c) holds, i.e., $(\lambda , \mu ) \in \text{Int} \, {\mathcal A} (3 , \Omega)$.
Then for $f \in \mathcal{CV}( \Omega )$, $(a_2(f), a_3(f)) = (\lambda , \mu)$
holds if, and only if, there exists $\omega^* \in H_1^\infty({\mathbb D})$ such that
$$
g(z) =  1 + \frac{zf''(z)}{f'(z)}
     =  P(z \sigma_{\gamma_1}(z \sigma_{\gamma_2}(z \omega^*(z))) ) .
$$
Let
\begin{equation}\label{eq:p-020}
 Q_{\gamma_1, \gamma_2}(z, \varepsilon ) =
 \int_0^z \zeta^{-1} \{ P( \zeta \sigma_{\gamma_1}(\zeta \sigma_{\gamma_2}
 ( \varepsilon \zeta )) ) -1 \} \, d \zeta , \quad
 z \in {\mathbb D} \; \text{and} \; \varepsilon \in \overline{\mathbb D} .
\end{equation}
Then for any fixed $\varepsilon \in \overline{\mathbb D}$, $Q_{\gamma_1, \gamma_2}(z, \varepsilon )$ is an
analytic function of $z \in {\mathbb D}$, and for each fixed $z \in {\mathbb D}$, $Q_{\gamma_1, \gamma_2}(z, \varepsilon )$ is an
analytic function of $\varepsilon \in \overline{\mathbb D}$. Thus by Theorem \ref{thm:Main_theorem} we have the following.

\begin{theorem}
Let $\Omega$ be a convex domain with $1 \in \Omega$ and $P$ be a conformal map of ${\mathbb D}$ onto
$\Omega$ with $P(z)=1+ \alpha_1 z + \cdots $.
Let $(\lambda , \mu ) \in {\mathbb C}^2$ and $\gamma_1 = \gamma_1(\lambda , \mu)$
and $\gamma_2 = \gamma_2(\lambda , \mu)$ defined by (\ref{eq:p-010}) and (\ref{eq:p-015}) respectively.
For $z_0 \in {\mathbb D} \backslash \{ 0 \}$, consider the variability region
$$
V_{\mathcal{CV}(\Omega)}(z_0,\lambda,\mu):=
\{ \log f'(z_0) : f \in \mathcal{CV}( \Omega ) \; \text{with} \; (a_2(f),a_3(f)) = (\lambda , \mu) \}.
$$
\begin{enumerate}[{\rm (i)}]

\item If $|\gamma_1(\lambda , \mu)|=1$ and $|\gamma_2(\lambda , \mu)| =0$,
then $V_{\mathcal{CV}(\Omega)}(z_0,\lambda,\mu)$
reduces to a set consists of a single point $w_0$, where
$w_0 = \int_0^{z_0} \zeta^{-1} \{ P(\gamma_1 \zeta)- 1 \} \, d \zeta$.\\

\item If $|\gamma_1(\lambda , \mu)| <1$ and $|\gamma_2(\lambda , \mu)| =1$,
then $V_{\mathcal{CV}(\Omega)}(z_0,\lambda,\mu)$
reduces to a set consists of a single point $w_0$, where
$w_0 = \int_0^{z_0} \zeta^{-1} \{ P(\zeta \sigma_{\gamma_1}( \gamma_2 \zeta))- 1 \} \, d \zeta$.\\

\item If $|\gamma_1(\lambda , \mu)| <1$ and $|\gamma_2(\lambda , \mu)| <1$,
i.e., $(\lambda , \mu ) \in \text{\rm Int} \, {\mathcal A}(3, \Omega)$
then $Q_{\gamma_1, \gamma_2}(z_0, \varepsilon )$ defined by (\ref{eq:p-020}) is a convex, univalent analytic function of
$\varepsilon \in \overline{\mathbb D}$ and
$$
V_{\mathcal{CV}(\Omega)}(z_0,\lambda,\mu) = Q_{\gamma_1, \gamma_2}(z_0, \overline{\mathbb D}) .
$$
Furthermore,
$\log f'(z_0) = Q_{\gamma_1, \gamma_2}(z_0, \varepsilon )$  for some $| \varepsilon | =1$ and
$f \in \mathcal{CV}( \Omega ) $, with $(a_2(f),a_3(f)) = (\lambda , \mu)$, if, and only if,
$$
  f(z)  =   \int_0^z  \exp  \left[  \int_0^{\zeta_1}  \zeta_2^{-1}
  \left\{ P ( z \sigma_{\gamma_1}(z \sigma_{\gamma_2}(\varepsilon \zeta_2 ))) - 1 \right\} \, d \zeta_2 \right] \, d \zeta_1 .
$$
\end{enumerate}
\end{theorem}

\begin{remark}
For a simply connected domain $\Omega$ with $1 \in \Omega$, define
$$
\mathcal{S}^*( \Omega ) =
 \left\{ f \in {\mathcal A}_0({\mathbb D} ) : z \frac{f'(z)}{f(z)} \in \Omega\; \text{in} \; {\mathbb D} \right\} .
$$
Then $f \in \mathcal{CV} (\Omega )$ if, and only if, $zf'(z) \in \mathcal{S}^*( \Omega )$. Thus we can easily translate the theorems of this section to results about variability regions of $\log \{ f(z_0)/z_0 \}$ when $f$ ranges over $\mathcal{S}^* ( \Omega )$ with or without the conditions $f''(0) = \lambda$ and $f'''(0) = \mu$.
\end{remark}

\subsection{Uniformly Convex Functions}

For $0\le k < \infty$, the class $k\mbox{-}\mathcal{UCV}$ of $k$-uniformly convex functions is defined
by $\mathcal{CV}(\Omega_k)$, where $\Omega_k:=\{ w \in {\mathbb C} : \text{Re} \, w > k |w-1| \}$.
Notice that $\Omega_k$ is the convex domain containing $1$, bounded by the conic section.
In this case, the conformal map $P_k$ which maps the unit disk $\mathbb{D}$ conformally onto $\Omega_k$
is given by
$$
P_k=
\begin{cases}
\frac{1}{1-k^2}\cosh\left(A\log\frac{1+\sqrt{z}}{1-\sqrt{z}}\right)-\frac{k^2}{1-k^2} & \mbox{for}\quad 0\le k<1\\[3mm]
1+\frac{2}{\pi^2}\left(\log\frac{1+\sqrt{z}}{1-\sqrt{z}}\right)^2 & \mbox{for}\quad k=1\\[3mm]
\frac{1}{k^2-1} \sin \left(\frac{\pi}{2K(x)} \int_0^{u(z)/\sqrt{x}} \frac{dt}{\sqrt{(1-t^2)(1-x^2 t^2)}}\right) + \frac{k^2}{k^2-1} & \mbox{for}\quad 1<k<\infty,
\end{cases}
$$
where $A=(2/\pi)\arccos k$, $u(z)=(z-\sqrt{x})/(1-\sqrt{x}z)$, and $K(x)$ is the elliptical integral defined by
$$
K(x)=\int_0^1 \frac{dt}{\sqrt{(1-t^2)(1-x^2 t^2)}},\quad x\in(0,1).
$$
For more details concerning uniformly convex functions, we refer to \cite{Kanas-Wisniowska-1999} and \cite{Ronning-1993}. When $k=0$, the class $0\mbox{-}\mathcal{UCV}$ is essentially the same as $\mathcal{CV}$. Let $P_k(z)= 1+\alpha_{k1}z+\alpha_{k2}z^2+\cdots$. Then it is a simple exercise to see that
$$
\alpha_{k1}=
\begin{cases}
\displaystyle \frac{2A^2}{1-k^2} & \mbox{for}\quad 0\le k<1\\[3mm]
\displaystyle 8/\pi^2 & \mbox{for}\quad k=1\\[1mm]
\displaystyle \frac{\pi^2}{4(k^2-1)K^2(x)(1+x)\sqrt{x}} & \mbox{for}\quad 1<k<\infty.
\end{cases}
$$
Let $f\in k\mbox{-}\mathcal{UCV}$ be of the form $f(z)=z+a_2z+a_3z^2+\cdots$ and $g(z)=1+zf''(z)/f'(z)$. Then from \eqref{eq:relation_between_c_and_b} and \eqref{eq:relation_between_b_and_a} we obtain $|a_2|\le \alpha_{k1}/2$. For $z_0 \in {\mathbb D} \backslash \{ 0 \}$ and $|\lambda|\le \alpha_{k1}/2$, consider the region of variability
$$
V_{k\mbox{-}\mathcal{UCV}}(z_0,\lambda)=
\{ \log f'(z_0) : f \in k\mbox{-}\mathcal{UCV} \; \mbox{with} \; a_2(f) = \lambda \}.
$$
The following corollary is a simple consequence of Theorem \ref{thm-p005}.

\begin{corollary}
Let $z_0 \in {\mathbb D} \backslash \{ 0 \}$ and $\lambda \in {\mathbb C}$ be such that $|\lambda | \leq \alpha_{1k}/2$. Also let $\gamma_1 = 2 \lambda /\alpha_{1k}$.
\begin{enumerate}[{\rm (i)}]

\item If $|\gamma_1 |= 1 $ then $V_{k\mbox{-}\mathcal{UCV}}(z_0,\lambda)=\{w_0\}$, where
$w_0 = \int_0^{z_0} \zeta^{-1} \{ P_k(\gamma_1 \zeta)- 1 \} \, d \zeta$.\\

\item If $|\gamma_1 |<1 $ then
$V_{k\mbox{-}\mathcal{UCV}}(z_0,\lambda) = Q_{\gamma_1}(z_0, \overline{\mathbb D} )$,
where
$$
 Q_{\gamma_1}(z_0, \varepsilon ) = \int_0^{z_0} \zeta^{-1} \left\{ P_k \left( \zeta
 \frac{\varepsilon \zeta + \gamma_1} {1+ \overline{\gamma_1} \varepsilon \zeta} \right) -1 \right\} \, d \zeta
$$
is a convex, univalent and analytic function of $\varepsilon \in \overline{\mathbb D}$. Furthermore
$$
 \log f'(z_0)  =  Q_{\gamma_1}(z_0, \varepsilon )
$$
 for some $\varepsilon \in \partial {\mathbb D}$ and $f \in k\mbox{-}\mathcal{UCV}$ with
$a_2(f) =  \lambda $, if, and only if,
$$
   f(z)  = \int_0^z  e^{Q_{\gamma_1}(\zeta , \varepsilon )} \, d \zeta , \quad z \in {\mathbb D}.
$$
\end{enumerate}
\end{corollary}

\subsection{Janowski Starlike and Convex Function}
For $A,B\in\mathbb{C}$ with $|B|\le 1$ and $A\ne B$, let $P_{A,B}(z):=(1+Az)/(1+Bz)$. Then $P_{A,B}$ is a conformal map of $\mathbb{D}$ onto a convex domain $\Omega_{A,B}$. In this case, the classes $\mathcal{S}^* ( \Omega_{A,B} )$ and $\mathcal{CV} (\Omega_{A,B})$ reduces to
$$
\mathcal{S}^*(A,B):=\left\{f \in {\mathcal A}_0({\mathbb D} ): \frac{z f'(z)}{f(z)} \prec \frac{1+A z}{1+B z} \right\}
$$
and
$$
\mathcal{CV}(A,B):=\left\{f \in {\mathcal A}_0({\mathbb D} ): \frac{z f''(z)}{f'(z)}+1 \prec \frac{1+A z}{1+B z} \right\}
$$
respectively. Since $P_{A,B}(\mathbb{D})=P_{-A,-B}(\mathbb{D})$, with out loss of generality we may assume that $A\in\mathbb{C}$ with $-1\le B\le 0$ and $A\ne B$. It is important to note that functions in $\mathcal{S}^*(A,B)$ with $A\in\mathbb{C}$, $-1\le B\le 0$ and $A\ne B$ are not  in general univalent. For $-1\le B<A\le 1$, it is easy to see that $\Omega_{A,B}\subset\mathbb{H}$, and so $\mathcal{S}^*(A,B)\subset\mathcal{S}^*$. Similar result holds for $\mathcal{CV}(A,B)$. (Note that for $-1\le B<A\le 1$, the class $\mathcal{S}^*(A,B)$ was first introduced and investigated by Janowski \cite{Janowski-1973}).

\bigskip
We also note that $P_{A,B}(z):=(1+Az)/(1+Bz)=1+(A-B)z+\cdots$. For $f\in\mathcal{CV}(A,B)$, from (\ref{eq:relation_between_c_and_b}) and (\ref{eq:relation_between_b_and_a}) we immediately obtain $|a_2(f)|\le |A-B|/2$. For $z_0 \in {\mathbb D} \backslash \{ 0 \}$ and $|\lambda|\le |A-B|/2$, consider the following
\begin{align*}
V_{\mathcal{CV}(A,B)}(z_0) &:= \{ \log f'(z_0 ) : f \in \mathcal{CV}(A,B) \},\\
V_{\mathcal{CV}(A,B)}(z_0,\lambda) &:=
\{ \log f'(z_0) : f \in \mathcal{CV}(A,B) \; \mbox{with} \; a_2(f) = \lambda \}.
\end{align*}
The following corrolary is a simple consequence of Theorems \ref{thm-p001} and \ref{thm-p005}.

\begin{corollary}\label{cor-p001}
Let $z_0 \in {\mathbb D} \backslash \{ 0 \}$ be fixed and $\lambda\in\mathbb{C}$ be such that $|\lambda|\le |A-B|/2$. Also let $\gamma_1 = 2 \lambda /(A-B)$.
\begin{enumerate}[{\rm (i)}]

\item The region of variability $V_{\mathcal{CV}(A,B)}(z_0)$
is a convex, closed, Jordan domain and coincides with the set $K( \overline{\mathbb D}(0,|z_0|))$, where
$$
K(z) = \int_0^z \frac{A-B}{1+B\zeta} \, d \zeta
$$
is a convex, univalent function in ${\mathbb D}$. Furthermore,
$\log f'(z_0) = K( \varepsilon z_0)$ for some $|\varepsilon | =1$ and $f \in \mathcal{CV}(A,B)$
if, and only if, $f(z) = \varepsilon^{-1}F(\varepsilon z)$, where $F(z)= \int_0^z e^{K(\zeta )} \, d \zeta $.\\

\item If $|\gamma_1 |= 1 $ then $V_{\mathcal{CV}(A,B)}(z_0,\lambda)=\{w_0\}$, where
$w_0 = \int_0^{z_0} \zeta^{-1} \{ P(\gamma_1 \zeta)- 1 \} \, d \zeta$.\\

\item If $|\gamma_1 |< 1 $ then
$V_{\mathcal{CV}(A,B)}(z_0,\lambda) = Q_{\gamma_1}(z_0, \overline{\mathbb D} )$,
where
$$
 Q_{\gamma_1}(z_0, \varepsilon ) = \int_0^{z_0}
 \frac{(A-B)\sigma_{\gamma_1}(\varepsilon \zeta)}{1+ B\zeta \sigma_{\gamma_1}(\varepsilon \zeta)} \, d \zeta
$$
is a convex, univalent and analytic function of $\varepsilon \in \overline{\mathbb D}$. Furthermore
$$
 \log f'(z_0)  =  Q_{\gamma_1}(z_0, \varepsilon )
$$
holds for some $\varepsilon \in \partial {\mathbb D}$ and $f \in \mathcal{CV}( \Omega )$ with
$a_2(f) =  \lambda $, if, and only if,
$$
   f(z)  = \int_0^z  e^{Q_{\gamma_1}(\zeta , \varepsilon )} \, d \zeta , \quad z \in {\mathbb D}.
$$
\end{enumerate}
\end{corollary}

\begin{remark}
The region of variability $V_{\mathcal{CV}(A,B)}(z_0,\lambda)$ for the class $\mathcal{CV}(A,B)$ was first obtained by Ul-Haq \cite{Ul-Haq-2014} for $-1\le B<0$ and $A>B$. Although, Ul-Haq considered the problem for $A\in\mathbb{C}$, $0<B\le 1$ and $A\ne B$, it is evident from the computation that it is valid only for $-1\le B<0$ and $A>B$. We  also note that the Herglotz representation (2) in \cite{Ul-Haq-2014} for functions in $\mathcal{CV}(A,B)$ is not valid when $-1< B<0$, and this was used by by Ul-Haq \cite{Ul-Haq-2014} in order to obtain the region of variability $V_{\mathcal{CV}(A,B)}(z_0)$.
\end{remark}

In particular, for $A= e^{-2i\alpha}$ with $\alpha\in(-\pi/2,\pi/2)$ and $B= -1$, the class $\mathcal{CV}(A,B)$ reduces to the class of functions which satisfy ${\rm Re\,} \{e^{i\alpha}(1+zf''(z)/f'(z))\}>0$ for $z\in\mathbb{D}$. This class, denoted by $\mathcal{S}_{\alpha}$  are known as Robertson functions. It is interesting to note that if we choose $A= e^{-2i\alpha}$ with $\alpha\in(-\pi/2,\pi/2)$ and $B= -1$ in Corollary \ref{cor-p001}. then we obtain the result obtained in \cite{Ponnusamy-Vasudevarao-Yanagihara-2008}.

\bigskip
For $A= 1-2\alpha$ with $-1/2\le\alpha<1$ and $B= -1$, the class $\mathcal{CV}(A,B)$ reduces to the class of functions $f$  satisfying  ${\rm Re\,} (1+zf''(z)/f'(z))>\alpha$ for $z\in\mathbb{D}$. This class, denoted by $\mathcal{CV}(\alpha)$  are known as convex functions of order $\alpha$. It is important to note that for $0\le\alpha<1$, $\mathcal{CV}(\alpha)\subset\mathcal{CV}$. On the other hand, for $-1/2\le\alpha<0$ functions in $\mathcal{CV}(\alpha)$ are convex functions in some direction (see \cite{Ponnusamy-Sahoo-Yanagihara-2014}). If we choose $A= 1-2\alpha$ with $-1/2\le\alpha<1$ and $B= -1$ in Corollary \ref{cor-p001} then we obtain the precise region of variability
$V_{\mathcal{CV}(\alpha)}(z_0):= \{ \log f'(z_0 ) : f \in \mathcal{CV}(\alpha) \}$ and
$V_{\mathcal{CV}(\alpha)}(z_0,\lambda):=\{ \log f'(z_0) : f \in \mathcal{CV}(\alpha) \; \mbox{and} \; a_2(f) = \lambda \}$,
which also gives a  generalization of Theorem \ref{theorem-A}. In particular, if we choose $A= 2$ and $B= -1$ in Corollary \ref{cor-p001} then we obtain the result obtained by Ponnusamy and Vasudevarao \cite[Theorem 2.6]{Ponnusamy-Vasudevarao-2007}. Similarly, for $A= -2$ and $B= -1$, the class $\mathcal{CV}(A,B)$ reduces to the class of functions $f$ which satisfy ${\rm Re\,} (1+zf''(z)/f'(z))<3/2$ for $z\in\mathbb{D}$. Note that functions in the class $\mathcal{CV}(-2,-1)$ are starlike, but not necessarily convex ( \cite{Ali-Vasudevarao-2015}), and if we choose $A= -2$ and $B= -1$ in Corollary \ref{cor-p001} then we obtain the result obtained in \cite[Theorem 2.8]{Ponnusamy-Vasudevarao-2007}.

\bigskip
Since $f \in \mathcal{CV}(A,B)$ if, and only if, $zf'(z) \in \mathcal{S}^*(A,B)$, we can easily translate the above results about variability regions of
$\log \{ f(z_0)/z_0 \}$ when $f$ ranges over $\mathcal{S}^*(A,B)$ with or without the condition $f''(0) = 2\lambda$.

\section{Concluding Remark}
Theorem \ref{thm:of_order_zero} demonstrates that our results are closely related to the concept of subordination. Our assumption $g \in {\mathcal F}_\Omega (c)$ in Theorem \ref{thm:Main_theorem} can be rewritten as $g \prec P$ when $c_0=0$. In this case $P^{-1} (g (z)) = c_1z+\cdots +c_n z^n + \cdots $.
However, apart from a few exceptional cases, we cannot express our conclusions
in terms of subordination relations. Let $c  = (c_0, \ldots , c_{n-1}) = (0,\ldots , 0)\in {\mathbb C}^n$.
Then the Schur parameter for $c$ is given by $\gamma = (\gamma_0, \ldots , \gamma_{n-1}) = (0, \ldots, 0)$.
For this particular choice of $c$, the function $Q_{\gamma , j}$ defined by (\ref{def:extremal_Q}) becomes
$$
 Q_{\gamma , j}(z, \varepsilon) =
 \int_0^z \zeta^j \{ P( \varepsilon \zeta^n ) -1 \}\, d \zeta .
$$
Let
$$
   H(z) =
 \frac{j+1}{z^{(j+1)/n}} \int_0^{z^{1/n}} \zeta^j \{ P( \zeta^n ) -1 \}\, d \zeta .
$$
Then
$$
 \frac{j+1}{z^{j+1}} Q_{\gamma , j}(z, \varepsilon ) = H( \varepsilon z^n ) .
$$
Since by Theorem \ref{thm:Main_theorem}, for each fixed $z \in {\mathbb D} \backslash \{ 0 \}$,
$Q_{\gamma , j}(z, \varepsilon )$ is a convex univalent function of $\varepsilon \in \overline{\mathbb D}$,
and $H( \varepsilon z^n )$ is also a convex univalent function of $\varepsilon \in \overline{\mathbb D}$.
By letting $z \rightarrow 1$ in ${\mathbb D}$, $H (\varepsilon )$ is also convex univalent in ${\mathbb D}$.
Let $g \in {\mathcal F}_\Omega$ with $g'(0)= \cdots = g^{(n-1)}(0) = 0$.
It follows from Theorem \ref{thm:Main_theorem} that for any $z \in {\mathbb D} \backslash \{ 0 \}$,
there exists $\varepsilon \in \overline{\mathbb D}$ satisfying
$$
 \int_0^z \zeta^j \{g(\zeta) -1 \} \, d \zeta = Q_{\gamma , j}(z, \varepsilon ) .
$$
Thus for all $z \in {\mathbb D}$, we have
$$
\frac{j+1}{z^{j+1}} \int_0^z \zeta^j \{g(\zeta) -1 \} \, d \zeta
= \frac{j+1}{z^{j+1}} Q_{\gamma , j}(z, \varepsilon )
= H( \varepsilon z^n ) \subset H( {\mathbb D }) .
$$
Consequently, in view of univalence of $H$ we obtain the following subordination relation
$$
  \frac{j+1}{z^{j+1}} \int_0^z \zeta^j \{g(\zeta) -1 \} \, d \zeta \prec H(z) .
$$
This was previously proved by Hallenbeck and Ruscheweyh \cite{Hallenbeck-Ruscheweyh-1975}. In fact, Hallenbeck and Ruscheweyh \cite{Hallenbeck-Ruscheweyh-1975} proved the above subordination relation when $\text{\rm Re} \, j \geq  -1$ with $j \not= -1$.

\vspace{1 cm}

\noindent{\bf Data availability statement:} There are no data associated with this article.

\end{document}